\documentclass[reqno]{amsart}
\usepackage{amsmath,amsthm,amssymb,amscd,float}
\usepackage{subcaption}
\usepackage[shortlabels]{enumitem}
\usepackage{youngtab}
\usepackage{tikz}
\usepackage{xcolor}

\newtheorem{theorem}{Theorem}[section]
\newtheorem{lemma}[theorem]{Lemma}

\theoremstyle{definition}
\newtheorem{definition}[theorem]{Definition}
\newtheorem{example}[theorem]{Example}

\numberwithin{equation}{section}
\allowdisplaybreaks
\begin{document}

\title[G\"ollnitz-Gordon-Andrews identities via commutative algebra]
 {On the G\"ollnitz-Gordon-Andrews identities via commutative algebra}

\author[Rupam Barman]{Rupam Barman}
\address{Department of Mathematics, Indian Institute of Technology Guwahati, Assam, India, PIN- 781039}
\email{rupam@iitg.ac.in}

 \author[Alapan Ghosh]{Alapan Ghosh}
 \address{Department of Mathematics, Indian Institute of Technology Guwahati, Assam, India, PIN- 781039}
 \email{alapan.ghosh@iitg.ac.in}
 
\author[Gurinder Singh]{Gurinder Singh}
\address{Postdoctoral Research Station of Mathematics, Hebei Normal University, Shijiazhuang 050024, P. R. China}
\email{gurindermaan1018@gmail.com}

\date{July 31, 2026}

\subjclass[2010]{11P81, 11P84, 13D40, 13A02}

\keywords{G\"ollnitz-Gordon-Andrews identities, partition identities, Hilbert-Poincar\'e series}

\begin{abstract} 
The G\"ollnitz-Gordon-Andrews identities generalize the classical partition identities discovered independently by H. G\"ollnitz and B. Gordon. These are Rogers-Ramanujan-type identities involving generating functions of partitions satisfying certain kinds of difference conditions on the one hand and infinite periodic products on the other. In 2021, Afsharijoo provided a commutative algebra proof of the Rogers-Ramanujan-Gordon identities. Building on Afsharijoo's approach, we investigate the G\"ollnitz-Gordon-Andrews identities using techniques from commutative algebra. More generally, we establish a broader family of identities, of which the G\"ollnitz-Gordon-Andrews identities arise as special cases. Our approach interprets the associated generating functions in terms of Hilbert-Poincar\'e series of suitably constructed graded algebras, providing the first commutative algebra framework for these identities.
\end{abstract}

\maketitle

\section{Introduction} 
A partition of a positive integer $n$ is a finite sequence of non-increasing positive integers $\lambda=(\lambda_1, \lambda_2, \ldots, \lambda_s)$ such that $\lambda_1+\lambda_2+\cdots +\lambda_s=n$. The integers $\lambda_j$ are called the parts of the partition $\lambda$. Let $p(n)$ denote the number of partitions of $n$, with the convention that $p(0):=1$. 
\par In 1748, Leonhard Euler \cite{Euler_1748} proved an elegant partition identity which says that the number of partitions of $n$ into odd parts is equal to the number of partitions of $n$ into distinct parts. Since then, numerous remarkable partition identities have been established (see, for example, \cite{Andrews_1998}). The partition identities studied in this article are the G\"ollnitz-Gordon-Andrews identities. In the 1960s, G\"ollnitz \cite{Gollnitz_1967} and Gordon \cite{Gordon_1965} independently discovered the following two identities. 
\begin{align*}
\sum_{n=0}^{\infty}\frac{q^{n^2}(1+q)(1+q^3)\cdots (1+q^{2n-1})}{(1-q^2)(1-q^4)\cdots  (1-q^{2n})}=\prod_{n=0}^{\infty}\frac{1}{(1-q^{8n+1})(1-q^{8n+4})(1-q^{8n+7})},\\
\sum_{n=0}^{\infty}\frac{q^{n^2+2n}(1+q)(1+q^3)\cdots (1+q^{2n-1})}{(1-q^2)(1-q^4)\cdots  (1-q^{2n})}=\prod_{n=0}^{\infty}\frac{1}{(1-q^{8n+3})(1-q^{8n+4})(1-q^{8n+5})}.
\end{align*}
The first identity yields that the number of partitions of any positive integer $n$ into parts congruent to $1, 4,$ or $7$ (mod $8$) is equal to the number of partitions $(\lambda_1,\lambda_2,\ldots,\lambda_s)$ of $n$ such that 
$$\lambda_m-\lambda_{m+1}\geq2~~\text{and}~~ \lambda_m-\lambda_{m+1}\geq3~~ \text{if}~ \lambda_m~\text{is~ even}.$$
The second identity gives that the number of partitions of any positive integer $n$ into parts congruent to $3,4$, or $5$ (mod $8$) is equal to the number of partitions $(\lambda_1,\lambda_2,\ldots,\lambda_s)$ of $n$ such that $\lambda_s\geq3$, $$\lambda_m-\lambda_{m+1}\geq2~~\text{and} ~~\lambda_m-\lambda_{m+1}\geq3 ~\text{if}~ \lambda_m~ \text{is~ even}.$$
\par Andrews \cite{Andrews_1967} generalized the G\"ollnitz-Gordon identities and we know it by the name of G\"ollnitz-Gordon-Andrews identities. It has been observed more recently that analytic versions of the original G\"ollnitz-Gordon identities already appear in Ramanujan's Lost Notebook \cite{RLNB}. For more information on these identities, see Section 1 of \cite{Coulson_2017}. We now state the G\"ollnitz-Gordon-Andrews identities. Let $r$ and $i$ be positive integers with $1\leq i\leq r$. Let $C_{r,i}(n)$ denote the number of partitions of $n$ into parts that are not congruent to $2$ (mod $4$) and are also not congruent to $0$ or $\pm(2i-1)$ (mod $4r$). Let $D_{r,i}(n)$ denote the number of partitions $(\lambda_1,\lambda_2,\ldots,\lambda_s)$ of $n$ satisfying the following conditions:
\begin{enumerate}[1.]
\item No odd part is repeated,
\item $\lambda_m-\lambda_{m+r-1}\geq2$ if $\lambda_m$ is odd,
\item $\lambda_m-\lambda_{m+r-1}\geq3$ if $\lambda_m$ is even, and
\item at most $i-1$ parts are equal to $1$ or $2$.
\end{enumerate} 
\begin{theorem}[G\"ollnitz-Gordon-Andrews identities]\label{Thm1}
Let $r$ and $i$ be positive integers with $1\leq i\leq r$. Then $$C_{r,i}(n)=D_{r,i}(n)~~ \text{for~ all}~ n\geq0.$$
\end{theorem}
The case $r=2$ in Theorem~\ref{Thm1} yields the classical G\"ollnitz-Gordon identities, while the case $r=1$ leads to the trivial identity $1=1$. Throughout this article, we assume $r\geq2$.
\par For $1\leq i\leq r$, define
\begin{align}\label{1.1}
\mathcal{C}_i(q):=\prod_{\substack{m\geq1,~m\not\equiv2\pmod4\\ m\not\equiv0,~2r\pm(2i-1)\pmod{4r}}}\frac{1}{(1-q^m)}.
\end{align}
Note that $\mathcal{C}_{r-i+1}(q)$ is the generating function for $C_{r,i}(n)$.
For $g\geq1$, the series $\mathcal{C}_{(r-1)g+i}(q)$ is defined recursively (see \cite[Section 2]{Coulson_2017}) as follows. For $i=1$,
\begin{align*}
\mathcal{C}_{(r-1)g+1}(q)=\mathcal{C}_{(r-1)(g-1)+r}(q),
\end{align*}
and for $i=2,\ldots,r$,
\begin{align*}
\mathcal{C}_{(r-1)g+i}(q)=\frac{\mathcal{C}_{(r-1)(g-1)+r-i+1}(q)-\mathcal{C}_{(r-1)(g-1)+r-i+2}(q)}{q^{2g(i-1)}}-\frac{\mathcal{C}_{(r-1)g+i-1}(q)}{q}.
\end{align*}
Let $r$, $i$, and $J\geq0$ be integers with $1\leq i\leq r$. Let $E_{r,i,J}(n)$ denote the number of partitions $(\lambda_1,\lambda_2,\ldots,\lambda_s)$ of $n$ satisfying the following conditions:
\begin{enumerate}[1.]
\item No odd part is repeated,
\item $\lambda_m-\lambda_{m+r-1}\geq2$ if $\lambda_m$ is odd,
\item $\lambda_m-\lambda_{m+r-1}\geq3$ if $\lambda_m$ is even,
\item all parts are greater than $2J$, and
\item at most $i-1$ parts are equal to $2J+1$ or $2J+2$.
\end{enumerate} 
Let $\mathcal{E}_{r,i,J}(q)$ denote the generating function for $E_{r,i,J}(n)$. The following theorem, due to Coulson \textit{et al.} \cite{Coulson_2017}, generalizes the G\"ollnitz-Gordon-Andrews identities.
\begin{theorem}\cite[Proposition 6.2]{Coulson_2017}\label{Thm2}
For any nonnegative integer $J$ and $1\leq i\leq r$, we have
\begin{align}\label{ll1}
\mathcal{C}_{(r-1)J+\ell}(q)=\mathcal{E}_{r,i,J}(q),
\end{align}
where $\ell=r-i+1$.
\end{theorem}
The case $J=0$ in Theorem~\ref{Thm2} gives the G\"ollnitz-Gordon-Andrews identities; hence Theorem~\ref{Thm1} is a direct corollary of Theorem~\ref{Thm2}. Note that $E_{r,i,0}(n)=D_{r,i}(n)$. For $J\geq 1$, however, no direct partition-theoretic interpretation of $\mathcal{C}_{(r-1)J+\ell}(q)$ is presently known, in contrast to the case $J=0$, where $\mathcal{C}_{\ell}(q)$ serves as the generating function of the partition function $C_{r,i}(n)$.
\par Recently, Afsharijoo and Mourtada \cite{Mourtada-new-1} introduced a commutative algebraic approach to partition identities via {G}r\"obner bases and Hilbert-Poincar\'{e} series. Their work provides the first commutative algebra proofs of the Rogers-Ramanujan identities and their extensions. Subsequently, Afsharijoo, Dousse, Jouhet, Mohsen, and Mourtada further developed and applied similar techniques from commutative algebra to the study of integer partitions; see, for example, \cite{Mourtada-new-2, Mourtada-new-3, Mourtada-new-4, Mourtada-new-5}. In 2021, Afsharijoo~\cite{Afsharijoo_2021} provided a commutative algebra proof of the Rogers-Ramanujan-Gordon identities. In \cite[Theorem 8.8]{Coulson_2017}, Coulson \textit{et al.} further introduced a generalization of these identities, referred to as the $J$-generalization, analogous to Theorem~\ref{Thm2}. In \cite{Ghosh-Barman}, the first and second authors extended Afsharijoo's approach to obtain a commutative algebra proof of the $J$-generalization of the Rogers-Ramanujan-Gordon identities. To the best of our knowledge, however, a commutative algebraic treatment of the G\"ollnitz-Gordon-Andrews identities has not yet been explored.
\par 
The main objective of this article is to provide a proof of the $J$-generalization of the G\"ollnitz-Gordon-Andrews identities using methods from commutative algebra. Our approach is inspired by the methods of Afsharijoo~\cite{Afsharijoo_2021} and Bruschek \emph{et al.}~\cite{Mourtada}. In particular, we analyze the identities in Theorem~\ref{Thm2} by establishing a connection between their partition-theoretic interpretations and the Hilbert-Poincar\'e series of suitably constructed graded algebras. This perspective yields the first commutative algebra framework for these identities.
\section{Preliminaries}\label{Section2}
In this section, we enlist some definitions and results from commutative algebra and topology. For more details, see for example \cite{Atiyah_book, Eisenbud, Pfister_book}.
\begin{definition}[Graded ring]
A graded ring is a ring $A$ together with a family $(A_j)_{j \geq 0}$ of subgroups of the additive group of $A$, such that $A=\bigoplus_{j=0}^{\infty} A_j$ and $A_{j_1}A_{j_2} \subseteq A_{j_1+j_2}$ for all $j_1,j_2 \geq 0$.
\end{definition}
Here $A_0$ is a subring of $A$ and each $A_j$ is an $A_0$-module. A non-zero element of $A_j$ is said to be homogeneous component of degree $j$.
\begin{definition}[Homogeneous ideal]
An ideal $I$ of a graded ring $A$ is called homogeneous if it is generated by homogeneous components.
\end{definition}
The intersection of a homogeneous ideal $I$ with $A_j$ is an $A_0$-submodule of $A_j$, called the homogeneous part of degree $j$ of $I$. A homogeneous ideal $I$ is the direct sum of its homogeneous parts $I_j=I \cap A_j$, i.e., $I=\bigoplus_{j=0}^{\infty} I_j$. If $I$ is a homogeneous ideal of a graded ring $A$, then the quotient ring $\frac{A}{I}$ is also a graded ring, decomposed as $$\frac{A}{I}=\bigoplus_{j=0}^{\infty} \frac{A_j}{I_j}.$$  
\begin{definition}[Graded $\mathbb{F}$-algebra]
Let $\mathbb{F}$ be a field. A graded ring $A=\bigoplus_{j=0}^{\infty} A_j$ is called a graded $\mathbb{F}$-algebra if it is also an $\mathbb{F}$-algebra, and $A_j$ is a vector space for all $j\geq0$ with $A_0=\mathbb{F}$.
\end{definition}
\begin{definition}[Weight of a polynomial]
The weight of the monomial $x_{i_1}^ {{\alpha}_1}\cdots x_{i_m}^{{\alpha}_m} \in \mathbb{F}[x_1,x_2,\ldots]$ is defined as $\sum_{k=1}^{m}i_k\alpha_k$. A polynomial $f(x) \in  \mathbb{F}[x_1,x_2,\ldots] $ is said to be a homogeneous polynomial of weight $a$ if every monomial of $f(x)$ has the same weight $a$.
\end{definition}
\begin{example}[Gradation by weight]\label{Example2}
Let $\mathbb{F}$ be a field of characteristic zero. Then $A:=\mathbb{F}[x_1,x_2,\ldots]$ is a graded algebra. $A$ is graded by weight, i.e., $A=\bigoplus_{j=0}^{\infty} A_j$, where $A_j$ is the set of polynomials of weight $j$ along with zero polynomial.
\end{example}
\begin{definition}[Hilbert-Poincar\'e series]
Let $\mathbb{F}$ be a field of characteristic zero and $A=\bigoplus_{j=0}^{\infty} A_j$ be a graded $\mathbb{F}$-algebra such that $\dim_{\mathbb{F}}(A_j)< \infty$. Then the Hilbert-Poincar\'e series of $A$ is $$\mathrm{HP}_A(q):=\sum_{j \geq 0}\dim_{\mathbb{F}}(A_j)q^j.$$
\end{definition}
Consider the graded algebra $A=\mathbb{F}[x_1,x_2,x_3,\ldots]$, graded by weight. Let $I$ be a homogeneous ideal of $A$, and let $f\in A$ be a homogeneous polynomial of weight $w$. Then, by \cite[Lemma 5.2.2]{Pfister_book}, we have
\begin{align}\label{HP_formula}
\mathrm{HP}_{\frac{A}{I}}(q)=q^w\mathrm{HP}_{\frac{A}{(I:f)}}(q)+\mathrm{HP}_{\frac{A}{(I,f)}}(q),
\end{align}
where $(I:f)=\{g\in A~\vert ~fg\in I\}$.  
\par We briefly recall some facts concerning the Krull topology. For further details, see, for example \cite{Eisenbud}. Let $I$ be an ideal of ring $A$. The Krull topology (or $I$-adic topology) on ring $A$ is defined by declaring a subset $U$ of $A$ to be open if, for every $x \in U$, there exists $j \in \mathbb{N}$ such that $x+I^j \subseteq U$. A sequence $(a_m)$ in $A$ converges to an element $a \in A$ if, for every $j \in \mathbb{N}$, there exists $N \in \mathbb{N}$ such that $a_m-a \in I^j$ for all $m \geq N$. In this article, we work with the $q$-adic topology, where $A=\mathbb{F}[[q]]$ and $I$ is the ideal generated by $q$ in $A$. 

\section{A Proof of Theorem~\ref{Thm2}}\label{Section3}
To prove the identities in Theorem~\ref{Thm2} for $r \geq 2$, we first relate the generating function of $E_{r,i,J}(n)$ to the Hilbert-Poincar\'e series of a suitable graded algebra. Let $\mathbb{F}$ be a field of characteristic zero. We consider the graded algebra $S:=\mathbb{F}[x_1,x_2,x_3,\ldots]$, where the gradation is by weight as described in Example~\ref{Example2}. For each integer $k \geq 1$, denote by $S_k:=\mathbb{F}[x_k,x_{k+1},x_{k+2},\ldots]$, so that $S_1=S$. We write $(S_k)_j$ for the homogeneous part of degree $j$ in $S_k$.
\par Let $a, b, c, n_1$, and $n_2$ be integers. For $r\geq2$, $1\leq i\leq r$, and a fixed nonnegative integer $J$, consider the ideal
\begin{align*}
&L_{r,i,J}:=\left(x_{2J+1}^2,~x_{2J+1}x_{2J+2}^{i-1},~x_{2J+2}^{i},~x_{2a-1}^2,~x_{2b-1}x_{2b}^{r-1},~x_{2c}^{r-n_1} x_{2c+2}^{n_1},\right.\\
&\left.x_{2c}^{r-n_2-1}x_{2c+1} x_{2c+2}^{n_2}:2a-1,2b-1,2c \geq 2J+2; 0 \leq n_1 \leq r-1; 0 \leq n_2 \leq r-2\right)
\end{align*}
of $S_{2J+1}$, which is readily seen to be homogeneous. Hence, the quotient $\frac{S_{2J+1}}{L_{r,i,J}}$ inherits a graded algebra structure. 
Moreover, for each $j\geq 0$, we have
\begin{align*}
\dim_{\mathbb{F}}\left(\frac{S_{2J+1}}{L_{r,i,J}}\right)_j=\dim_{\mathbb{F}}\left(\frac{(S_{2J+1})_j}{(L_{r,i,J})_j}\right)\leq\dim_{\mathbb{F}}((S_{2J+1})_j) \leq \dim_{\mathbb{F}}((S)_j)= p(j)<\infty.
\end{align*}
This establishes the existence of Hilbert-Poincar\'e series of $\frac{S_{2J+1}}{L_{r,i,J}}$, which is defined by 
\begin{align*}
\mathrm{HP}_{\frac{S_{2J+1}}{L_{r,i,J}}}(q)=\sum_{j\geq0}\dim_{\mathbb{F}}\left(\frac{S_{2J+1}}{L_{r,i,J}}\right)_j q^j.
\end{align*}
We now relate the Hilbert-Poincar\'e series $\mathrm{HP}_{\frac{S_{2J+1}}{L_{r,i,J}}}(q)$ to the partition function $E_{r,i,J}(n)$. By construction, the ideal $L_{r,i,J}$ is generated by $x_{2J+1}^2$, $x_{2J+1}x_{2J+2}^{i-1}$, $x_{2J+2}^{i}$ and monomials of the form  
$$x_{2a-1}^2, \quad x_{2b-1}x_{2b}^{r-1}, \quad x_{2c}^{r-n_1}x_{2c+2}^{n_1}, \quad x_{2c}^{r-n_2-1}x_{2c+1} x_{2c+2}^{n_2},$$ 
where $2a-1,2b-1,2c \geq 2J+2$, $0 \leq n_1 \leq r-1$, and $0 \leq n_2 \leq r-2 $. Observe that the graded component $\left(\frac{S_{2J+1}}{L_{r,i,J}}\right)_j$ is generated by monomials of the form $x_{l_1}x_{l_2} \cdots x_{l_m} \in S_{2J+1}/L_{r,i,J}$, of weight $\sum_{p=1}^{m}l_p=j$. To each such monomial, we associate the partition $(l_1,l_2,\ldots,l_m)$ of $j$, which satisfies the defining conditions of $E_{r,i,J}(j)$. This correspondence is bijective, and hence  $$\dim_{\mathbb{F}}\left(\frac{S_{2J+1}}{L_{r,i,J}}\right)_j=E_{r,i,J}(j).$$
Consequently,
\begin{align}\label{eqj1}
\mathrm{HP}_{\frac{S_{2J+1}}{L_{r,i,J}}}(q)=\sum_{j\geq0}E_{r,i,J}(j)q^j.
\end{align}
Let $a,b,c,n_1,n_2$, and $\ell $ be integers with $1 \leq \ell \leq r$. We now define the following two ideals of $S_k$ for $k\geq2J+1$:
\begin{align*}
    L_{k}:=&\left(x_{2a-1}^2, x_{2b-1}x_{2b}^{r-1}, x_{2c}^{r-n_1} x_{2c+2}^{n_1}\right., x_{2c}^{r-n_2-1}x_{2c+1} x_{2c+2}^{n_2}:\\
    &\left.2a-1,2b-1,2c \geq k,0 \leq n_1 \leq r-1,~\text{and}~0 \leq n_2 \leq r-2 \right)
\end{align*}
and
\begin{align*}
L_{k}^{\ell}:=\begin{cases}
	\left(x_{k}^{\ell}, x_{k}^{\ell-1}x_{k+2}^{r-\ell+1},x_{k}^{\ell-2}x_{k+2}^{r-\ell+2},\ldots,x_{k}x_{k+2}^{r-1},x_{k}^{\ell-1}x_{k+1}x_{k+2}^{r-\ell}\right.,\\  
    \left.x_{k}^{\ell-2}x_{k+1}x_{k+2}^{r-\ell+1}, \ldots,x_{k}x_{k+1}x_{k+2}^{r-2},L_{k+1}\right) &\text{ if } k~ \text {is even}; \\
	\left(x_k^2,x_{k}x_{k+1}^{\ell-1},L_{k+1}^{\ell} \right) &\text{ if } k~ \text {is odd}.
\end{cases}
\end{align*}
\par We denote the Hilbert-Poincar\'e series $\mathrm{HP}_{\frac{S_{k}}{L_{k}}}(q)$ by $\mathrm{HP}^{k}$ and the Hilbert-Poincar\'e series $\mathrm{HP}_{\frac{S_{k}}{L_{k}^{\ell}}}(q)$ by $\mathrm{HP}_{\ell}^{k}$. Also, we use $\mathrm{HP}\left(\frac{A}{I}\right)$ in place of $\mathrm{HP}_{\frac{A}{I}}(q)$. With this notation, we record the following:
\begin{enumerate}
\item[(N1)] $\mathrm{HP}_1^{k}= \begin{cases}
\mathrm{HP}^{k+1}      &\text{ if } k~ \text {is even};\\
\mathrm{HP}^{k+2}      &\text{ if } k~ \text {is odd}.
\end{cases}$
\item[(N2)] $\mathrm{HP}_r^{k}=\mathrm{HP}^{k}$.
\item[(N3)] $\mathrm{HP}\left(\frac{S_{2J+1}}{L_{r,i,J}}\right)=\mathrm{HP}\left(\frac{S_{2J+1}}{L_{2J+1}^{i}}\right)=\mathrm{HP}_{i}^{2J+1}$.
\end{enumerate}
In view of (N3) and \eqref{eqj1}, it follows that
\begin{align}\label{eqj2}
\mathrm{HP}_{i}^{2J+1}=\sum_{n\geq0}E_{r,i,J}(n)q^n.
\end{align}
\par Now, we prove a recursion formula for $\mathrm{HP}_{i}^{2J+1}$. In fact, we prove a recursion formula for $\mathrm{HP}_{\ell}^{k}$, for odd positive integers $k\geq2J+1$ in the following lemma.
\begin{lemma}\label{Lemma2.1}
Let $J$ be a nonnegative integer, and let $k,r$, and $\ell$ be positive integers with $r \geq 2$ and $1\leq \ell \leq r$. Then, for any odd $k \geq 2J+1$, we have 
\begin{align}\label{n1}
\mathrm{HP}_\ell ^{k}= \sum_{j=1}^{\ell-1} q^{(k+1)j-1} \mathrm{HP}_{r-j+1} ^{k+2}+\sum_{j=1}^{\ell} q^{(k+1)(j-1)} \mathrm{HP}_{r-j+1} ^{k+2}. 
\end{align}
\end{lemma}
\begin{proof}
In this proof, we repeatedly apply \eqref{HP_formula} without explicit reference. Let $k \geq 2J+1$ be an odd integer and $\ell=1$. By (N1) and (N2), we obtain
\begin{align}\label{f1}
\mathrm{HP}_{\ell}^{k}=\mathrm{HP}_1^{k}=\mathrm{HP}^{k+2}=\mathrm{HP}^{k+2}_{r}.
\end{align}
Next, let $k \geq 2J+1$ be odd and $\ell \neq 1$. Then, we have
\begin{align}
		\mathrm{HP}_{\ell}^{k}&=q^k \mathrm{HP} \left(\frac{S_k}{\left(L_{k}^{\ell}:x_k \right)}\right)+ \mathrm{HP} \left(\frac{S_k}{\left(L_{k}^{\ell},x_k \right)}\right)\nonumber \\
		&=q^k \mathrm{HP} \left(\frac{S_k}{\left(x_k,x_{k+1}^{\ell-1}, L_{k+1}^{\ell} \right)}\right) +\mathrm{HP} \left(\frac{S_{k+1}}{L_{k+1}^{\ell}}\right) \nonumber  \\
		&=q^k \mathrm{HP} \left(\frac{S_k}{\left(x_k,L_{k+1}^{\ell-1} \right)}\right)+ \mathrm{HP}_{\ell}^{k+1}\nonumber \\
		&=q^k \mathrm{HP} \left(\frac{S_{k+1}}{L_{k+1}^{\ell-1}}\right)+\mathrm{HP}_{\ell}^{k+1}\nonumber \\
		&=q^k \mathrm{HP}_{\ell-1}^{k+1}+\mathrm{HP}_{\ell}^{k+1}. \label{f2}
\end{align}
Now, suppose $k\geq 2J+1$ is even. A repeated application of \eqref{HP_formula} yields
\begin{align}
\mathrm{HP}_{\ell}^{k}&=q^k \mathrm{HP} \left(\frac{S_k}{\left(L_{k}^{\ell}:x_k \right)}\right)+ \mathrm{HP} \left(\frac{S_k}{\left(L_{k}^{\ell},x_k \right)}\right) \nonumber\\
&=q^k \mathrm{HP} \left(\frac{S_k}{\left(x_k^{\ell-1},x_{k}^{\ell-2}x_{k+2}^{r-\ell+1},\ldots,x_kx_{k+2}^{r-2}, x_{k}^{\ell-2}x_{k+1}x_{k+2}^{r-\ell},\ldots, x_kx_{k+1}x_{k+2}^{r-3}, L_{k+1}^{r-1} \right)}\right)\nonumber\\ 
&+ \mathrm{HP} \left(\frac{S_k}{\left(L_{k+1},x_k \right)}\right)
\nonumber\\
&=q^{2k} \mathrm{HP} \left(\frac{S_k}{\left(x_k^{\ell-2},x_{k}^{\ell-3}x_{k+2}^{r-\ell+1},\ldots,x_kx_{k+2}^{r-3}, x_{k}^{\ell-3}x_{k+1}x_{k+2}^{r-\ell},\ldots, x_kx_{k+1}x_{k+2}^{r-4}, L_{k+1}^{r-2} \right)}\right)\nonumber\\ 
&+  q^k \mathrm{HP} \left(\frac{S_k}{\left(x_k,L_{k+1}^{r-1} \right)}\right)+\mathrm{HP} \left(\frac{S_{k+1}}{L_{k+1}}\right) \nonumber\\
&=q^{2k} \mathrm{HP} \left(\frac{S_k}{\left(x_k^{\ell-2},x_{k}^{\ell-3}x_{k+2}^{r-\ell+1},\ldots,x_kx_{k+2}^{r-3}, x_{k}^{\ell-3}x_{k+1}x_{k+2}^{r-\ell},\ldots, x_kx_{k+1}x_{k+2}^{r-4}, L_{k+1}^{r-2} \right)}\right)\nonumber\\ 
&+ q^k \mathrm{HP}_{r-1}^{k+1}+\mathrm{HP}^{k+1}\nonumber\\
&=\cdots = \sum_{j=1}^{\ell}q^{k(j-1)}\mathrm{HP}_{r-j+1}^{k+1}.\label{f3}
\end{align}
Finally, since $k$ is odd, $k+1$ is even, hence we can substitute \eqref{f3} for $\mathrm{HP}_{\ell}^{k+1}$ in \eqref{f2}. For odd $k\geq 2J+1$ and $\ell \neq 1$, we obtain
\begin{align}\label{new-eqn-01}
\mathrm{HP}_{\ell}^{k}&=q^k \mathrm{HP}_{\ell-1}^{k+1}+\mathrm{HP}_{\ell}^{k+1} \nonumber\\
		             &=q^k \left(   \sum_{j=1}^{\ell-1}q^{(k+1)(j-1)}\mathrm{HP}_{r-j+1}^{k+2} \right)
		             + \left( \sum_{j=1}^{\ell}q^{(k+1)(j-1)}\mathrm{HP}_{r-j+1}^{k+2} \right) \nonumber\\
		             &= \sum_{j=1}^{\ell-1}q^{(k+1)j-1}\mathrm{HP}_{r-j+1}^{k+2} +  \sum_{j=1}^{\ell}q^{(k+1)(j-1)}\mathrm{HP}_{r-j+1}^{k+2}.
\end{align}
Combining \eqref{new-eqn-01} with \eqref{f1}, we obtain the recursion formula \eqref{n1} for $\mathrm{HP}_{\ell}^{k}$ for all odd $k \geq 2J+1$ and $1\leq \ell \leq r$.
\end{proof}
In the following lemma, we provide a recursion formula for $\mathrm{HP}_i ^{2J+1}$.
\begin{lemma}\label{Lemma2.2}
Let $J$ be a nonnegative integer, and let $r, i$ be integers with $r \geq 2$ and $1 \leq i \leq r$. Then the following recursion formula holds:
\begin{align}\label{t1}
\mathrm{HP}_i ^{2J+1}= \sum_{j=1}^{r} N_{i,j,(r-1)d+j}^{J} \mathrm{HP}_{r-j+1} ^{2d+1},
\end{align}
where $d \geq J+1$. The coefficients $N_{i,j,(r-1)d+j}^{J} \in \mathbb{F}[[q]]$ satisfy the recursion 
$$N_{i,j,(r-1)(d+1)+j}^{J}= q^{2(d+1)(j-1)} \sum_{m=1}^{r-j+1} N_{i,m,(r-1)d+m}^{J}+ q^{2(d+1)j-1} \sum_{m=1}^{r-j} N_{i,m,(r-1)d+m}^{J}$$
for $1 \leq j \leq r$, together with the initial conditions $($corresponding to $d=J+1$$)$:
$$N_{i,j,(r-1)(J+1)+j}^{J}=\begin{cases}
q^{2(J+1)j-1}+q^{2(J+1)(j-1)} &\text{ if } 1\leq j \leq i-1; \\
q^{2(J+1)(j-1)} &\text{ if } j=i;  \\
0 &\text{ if } i+1 \leq j \leq r.
\end{cases}$$
\end{lemma}
\begin{proof}
To prove the required recursion formula \eqref{t1}, we proceed by induction on $d$. We first verify the case $d=J+1$. Applying Lemma~\ref{Lemma2.1} with $k=2J+1$ and $\ell=i$, we obtain
\begin{align*}
\mathrm{HP}_{i}^{2J+1}&=\sum_{j=1}^{i-1}q^{(2J+2)j-1}\mathrm{HP}_{r-j+1}^{2J+3} +  \sum_{j=1}^{i}q^{(2J+2)(j-1)}\mathrm{HP}_{r-j+1}^{2J+3}\\
&=q^{(2J+2)(i-1)}\mathrm{HP}_{r-i+1}^{2J+3}+\sum_{j=1}^{i-1}\left(q^{(2J+2)j-1}+q^{(2J+2)(j-1)}\right)\mathrm{HP}_{r-j+1}^{2J+3} \\
&=\sum_{j=1}^{r} N_{i,j,(r-1)(J+1)+j}^{J}\mathrm{HP}_{r-j+1} ^{2J+3}.
\end{align*}
Thus, \eqref{t1} holds for $d=J+1$. Next, we assume that \eqref{t1} holds for all integers $d$ with $J+1 \leq d \leq s$. In particular, for $d=s$, we have
$$\mathrm{HP}_{i}^{2J+1}=\sum_{j=1}^{r} N_{i,j,(r-1)s+j}^{J}\mathrm{HP}_{r-j+1} ^{2s+1}.$$
We now prove \eqref{t1} for $d=s+1$. Substituting the expression for $\mathrm{HP}_{r-j+1} ^{2s+1}$ from Lemma~\ref{Lemma2.1} into the above equation, we obtain
\begin{align*}
\mathrm{HP}_{i}^{2J+1}= \sum_{j=1}^{r} N_{i,j,(r-1)s+j}^{J} \left( \sum_{m=1}^{r-j}q^{2(s+1)m-1}\mathrm{HP}_{r-m+1}^{2s+3} +  \sum_{m=1}^{r-j+1}q^{2(s+1)(m-1)}\mathrm{HP}_{r-m+1}^{2s+3}        \right). 
\end{align*}
Rearranging the sums, we obtain
\begin{align*}
\mathrm{HP}_{i}^{2J+1}=&\sum_{\ell=1}^{r} \left( q^{2(s+1)(r-\ell)} \sum_{j=1}^{\ell} N_{i,j,(r-1)s+j}^{J}\right.\\
&\left.+ q^{2(s+1)(r-\ell+1)-1} \sum_{j=1}^{\ell-1} N_{i,j,(r-1)s+j}^{J}         \right)\mathrm{HP}_{\ell} ^{2(s+1)+1} \\
=&\sum_{\ell=1}^{r} N_{i,r-\ell+1,(r-1)(s+1)+(r-\ell+1)}^{J} \mathrm{HP}_{\ell} ^{2(s+1)+1} \\
=&\sum_{j=1}^{r} N_{i,j,(r-1)(s+1)+j}^{J}\mathrm{HP}_{r-j+1} ^{2(s+1)+1}.
\end{align*}
This establishes \eqref{t1} for $d=s+1$, and hence completes the induction. Therefore, the recursion formula \eqref{t1} holds for all $d\geq J+1$. 
\end{proof}
Next, we write a recursion formula for $\mathcal{C}_{(r-1)J+\ell}$, which is given by Coulson \textit{et al.} \cite{Coulson_2017}.
\begin{lemma}[\cite{Coulson_2017}]\label{Lemma2.3}
Let $J$ be a nonnegative integer and  $r,\ell$ be integers with $r \geq 2,1 \leq \ell \leq r$. Then for $d \geq J+1$ we have the following recursion formula
$$\mathcal{C}_{(r-1)J+\ell}=\sum_{j=1}^{r} M_{\ell,j,(r-1)d+j}^J \mathcal{C}_{(r-1)d+j}.$$
Here, the coefficients $M_{\ell,j,(r-1)d+j}^J \in \mathbb{F}[[q]]$ satisfy the following recursion formula for $1 \leq j \leq r$  
$$M_{\ell,j,(r-1)(d+1)+j}^J= q^{2(d+1)(j-1)} \sum_{m=1}^{r-j+1} M_{\ell,m,(r-1)d+m}^J+ q^{2(d+1)j-1} \sum_{m=1}^{r-j} M_{\ell,m,(r-1)d+m}^J$$
with the following initial conditions (for $d=J+1$)
$$M_{\ell,j,(r-1)(J+1)+j}^J=\begin{cases}
q^{2(J+1)j-1}+q^{2(J+1)(j-1)} &\text{ if } 1\leq j \leq r-\ell; \\
q^{2(J+1)(j-1)} &\text{ if } j=r-\ell+1;  \\
0 &\text{ if } r-\ell+2 \leq j \leq r.
\end{cases}$$
\end{lemma}
\begin{proof}
The recursion formula for $\mathcal{C}_{(r-1)J+\ell}$ follows from (5.9) of \cite{Coulson_2017}, with $G$ replaced by $\mathcal{C}$. The recursion formula for $M^J_{\ell,j,(r-1)d+j}$ is given in \cite[Proposition 5.1]{Coulson_2017}. In our notation, the quantity $_i ^{J} h_{l}^{(j)}$ appearing in \cite[Proposition 5.1]{Coulson_2017} is identical with $M_{\ell,j,(r-1)d+j}^{J}$ under the substitutions $j\mapsto d$, $i\mapsto \ell$, and $l\mapsto j$.
\end{proof}
To prove Theorem~\ref{Thm2}, it suffices to show that $\mathrm{HP}^{2J+1}_{i}=\mathcal{C}_{(r-1)J+\ell}$ (see \eqref{eqj2}), where $\ell=r-i+1$. To this end, we first verify that the coefficients appearing in the recursion formulas for $\mathrm{HP}^{2J+1}_{i}$ (Lemma~\ref{Lemma2.2}) and $\mathcal{C}_{(r-1)J+\ell}$ (Lemma~\ref{Lemma2.3}) coincide.
\begin{lemma}\label{Lemma2.4}
For all $d \geq J+1$, $r \geq 2$, and $1 \leq v \leq r$, we have 
\begin{align}\label{neq1}
M_{\ell,v,(r-1)d+v}^J=N_{i,v,(r-1)d+v}^J,
\end{align}
where $\ell=r-i+1$.
\end{lemma}
\begin{proof}
We prove \eqref{neq1} by induction on $d$. For the base case $d=J+1$, Lemma~\ref{Lemma2.3} gives
$$M_{\ell,v,(r-1)(J+1)+v}^{J}=\begin{cases}
q^{2(J+1)v-1}+q^{2(J+1)(v-1)} &\text{ if } 1\leq v \leq r-\ell; \\
q^{2(J+1)(v-1)} &\text{ if } v=r-\ell+1;  \\
0 &\text{ if } r-\ell+2 \leq v \leq r.
\end{cases}$$
Substituting $\ell=r-i+1$, we obtain
$$M_{r-i+1,v,(r-1)(J+1)+v}^{J}=\begin{cases}
q^{2(J+1)v-1}+q^{2(J+1)(v-1)} &\text{ if } 1\leq v \leq i-1; \\
q^{2(J+1)(v-1)} &\text{ if } v=i;  \\
0 &\text{ if } i+1 \leq v \leq r.
\end{cases}$$
Comparing this with Lemma~\ref{Lemma2.2}, we conclude that 
\begin{align*}
M_{\ell,v,(r-1)(J+1)+v}^{J}=N_{i,v,(r-1)(J+1)+v}^{J},
\end{align*}
for $\ell=r-i+1$. This establishes \eqref{neq1} for $d=J+1$.
\par Now assume that \eqref{neq1} holds for all $J+1 \leq d\leq s$. In particular, for $d=s$ and $\ell=r-i+1$, we have
\begin{align}\label{neq2}
M_{\ell,v,(r-1)s+v}^{J}=N_{i,v,(r-1)s+v}^{J}. 
\end{align}
By Lemma~\ref{Lemma2.3}, 
$$M_{\ell,v,(r-1)(s+1)+v}^{J}= q^{2(s+1)(v-1)} \sum_{m=1}^{r-v+1} M_{\ell,m,(r-1)s+m}^{J}+ q^{2(s+1)v-1} \sum_{m=1}^{r-v} M_{\ell,m,(r-1)s+m}^{J}.$$
Substituting \eqref{neq2} into this expression, we obtain
\begin{align*}
M_{\ell,v,(r-1)(s+1)+v}^{J}&= q^{2(s+1)(v-1)} \sum_{m=1}^{r-v+1} N_{i,m,(r-1)s+m}^{J}+ q^{2(s+1)v-1} \sum_{m=1}^{r-v} N_{i,m,(r-1)s+m}^{J}\\
&=N_{i,v,(r-1)(s+1)+v}^{J},
\end{align*}
where the last equality follows from Lemma~\ref{Lemma2.2}. Thus, \eqref{neq1} holds for $d=s+1$, completing the induction. Hence, \eqref{neq1} is valid for all $d\geq J+1$.	
\end{proof}
We are now in a position to prove Theorem~\ref{Thm2}. In the theorem below, we show that $\mathcal{C}_{(r-1)J+\ell}=\mathrm{HP}^{2J+1}_{i}$, which immediately yields Theorem~\ref{Thm2}.
\begin{theorem}\label{Theorem2.5}
For $r \geq 2$, $1 \leq i \leq r$, and $J\geq0$, we have 
$$\mathcal{C}_{(r-1)J+\ell}=\mathrm{HP}_i^{2J+1},$$ 
where $\ell=r-i+1$.
\end{theorem}
\begin{proof}
From Lemma~\ref{Lemma2.3}, we have
\begin{align*}
M_{\ell,m,(r-1)(d+1)+m}^{J}&= q^{2(d+1)(m-1)} \sum_{t=1}^{r-m+1} M_{\ell,t,(r-1)d+t}^{J}+ q^{2(d+1)m-1} \sum_{t=1}^{r-m} M_{\ell,t,(r-1)d+t}^{J}\\
&=q^{2(d+1)(m-1)} \left(  \sum_{t=1}^{r-m+1} M_{\ell,t,(r-1)d+t}^{J}+ q^{2d+1} \sum_{t=1}^{r-m} M_{\ell,t,(r-1)d+t}^{J}\right).
\end{align*}
Clearly, the $q$-adic valuation of $M_{\ell,m,(r-1)(d+1)+m}^{J}$ is at least $2(d+1)(m-1)$. Hence, for each $1\leq m \leq r$, the limit 
$$\lim_{d \to +\infty}  M_{\ell,m,(r-1)(d+1)+m}^{J}$$ exists in the $q$-adic topology (see Section~\ref{Section2}). In particular, for $2\leq m \leq r$, we have
\begin{align*}
\lim_{d \to +\infty} & M_{\ell,m,(r-1)(d+1)+m}^{J}\\
&= \lim_{d \to +\infty} \left(  q^{2(d+1)(m-1)} \left(  \sum_{t=1}^{r-m+1} M_{\ell,t,(r-1)d+t}^{J}+ q^{2d+1} \sum_{t=1}^{r-m} M_{\ell,t,(r-1)d+t}^{J}    \right)  \right)\\
&=0.
\end{align*}
By Lemma~\ref{Lemma2.3}, we have 
$$\mathcal{C}_{(r-1)J+\ell}=\sum_{j=1}^{r} M_{\ell,j,(r-1)(d+1)+j}^{J} \mathcal{C}_{(r-1)(d+1)+j}.$$
Passing to the limit as $d \to +\infty$ on both sides of the above equation, we obtain
\begin{align}\label{g1}
\mathcal{C}_{(r-1)J+\ell}= \lim_{d \to +\infty} M_{\ell,1,(r-1)(d+1)+1}^{J} \mathcal{C}_{(r-1)(d+1)+1}.
\end{align}
By \cite[Theorem 4.1]{Coulson_2017}, we find that
\begin{align}\label{g2}
 \lim_{d \to +\infty} \mathcal{C}_{(r-1)(d+1)+1}=1.  
\end{align}
Combining \eqref{g1} and \eqref{g2}, we deduce 
$$\mathcal{C}_{(r-1)J+\ell}=\lim_{d \to +\infty} M_{\ell,1,(r-1)(d+1)+1}^{J}. $$
As mentioned earlier, the limit $\lim_{d \to +\infty} M_{\ell,1,(r-1)(d+1)+1}^{J}$ exists. We denote this limit by $M^J_{\ell,1,\infty}$. Thus,
\begin{align}\label{g3}
\mathcal{C}_{(r-1)J+\ell}=M_{\ell,1,\infty}^{J}.
\end{align}
\par Next, we consider the coefficients in the recursion formula of $\mathrm{HP}^{2J+1}_{i}$ under a similar situation. By Lemma~\ref{Lemma2.2}, we have
\begin{align*}
N_{i,m,(r-1)(d+1)+m}^{J}&= q^{2(d+1)(m-1)} \sum_{t=1}^{r-m+1} N_{i,t,(r-1)d+t}^{J}+ q^{2(d+1)m-1} \sum_{t=1}^{r-m} N_{i,t,(r-1)d+t}^{J}\\
&=q^{2(d+1)(m-1)} \left(  \sum_{t=1}^{r-m+1} N_{i,t,(r-1)d+t}^{J}+ q^{2d+1} \sum_{t=1}^{r-m} N_{i,t,(r-1)d+t}^{J}    \right).
\end{align*}
A completely analogous argument shows that the limit $$ \lim_{d \to +\infty}  N_{i,m,(r-1)(d+1)+m}^{J}$$ exists for all $1\leq m \leq r$ and vanishes for $2\leq m \leq r$. Moreover, for $d \geq J+1$,  Lemma~\ref{Lemma2.2} yields
$$\mathrm{HP}_i ^{2J+1}= \sum_{j=1}^{r} N_{i,j,(r-1)(d+1)+j}^{J}\mathrm{HP}_{r-j+1} ^{2d+3}.$$
Passing to the limit as $d \to +\infty$, we obtain
\begin{align}\label{g4}
\mathrm{HP}_i ^{2J+1}&= \lim_{d \to +\infty} N_{i,1,(r-1)(d+1)+1}^{J}\mathrm{HP}_{r} ^{2d+3}\nonumber\\
&= \lim_{d \to +\infty} N_{i,1,(r-1)(d+1)+1}^{J} \mathrm{HP}^{2d+3}.
\end{align}
We claim that $\lim_{d \to +\infty}\mathrm{HP}^{2d+3}=1$. We note that
$$\mathrm{HP} ^{2d+3}=\mathrm{HP} \left( \frac{S_{2d+3}}{L_{2d+3}}\right).$$ 
In the graded algebra $\frac{S_{2d+3}}{L_{2d+3}}$, the degree-zero homogeneous component is isomorphic to $\mathbb{F}$, and hence one-dimensional. For $1 \leq u \leq 2d+2$, the homogeneous component of degree $u$ is zero as there are no monomials of weight between $1$ and $2d+2$. Hence, $\mathrm{HP} ^{2d+3}=1+q^{2d+3}f(q)$ for some $f(q) \in \mathbb{F}[[q]]$, which implies
\begin{align}\label{g5}
\lim_{d \to +\infty}\mathrm{HP}^{2d+3}=1.
\end{align}
Substituting \eqref{g5} into \eqref{g4}, we obtain
$$\mathrm{HP}_i ^{2J+1}=\lim_{d \to +\infty} N_{i,1,(r-1)(d+1)+1}^{J}.$$
As seen earlier, the limit $\lim_{d \to +\infty} N_{i,1,(r-1)(d+1)+1}^{J}$ exists. We denote this limit by $N_{i,1,\infty}^{J}$. Thus,
\begin{align}\label{g6}
\mathrm{HP}_i ^{2J+1}=N_{i,1,\infty}^{J}.
\end{align}
Finally, by Lemma~\ref{Lemma2.4}, for $1\leq m \leq r$, we have
\begin{align*}
M_{\ell,m,(r-1)(d+1)+m}^{J}=N_{i,m,(r-1)(d+1)+m}^{J},
\end{align*}
where $\ell=r-i+1$. Passing to the limit as $d \to +\infty$ yields
\begin{align}\label{g7}
M_{\ell,1,\infty}^{J}=N_{i,1,\infty}^{J}.
\end{align}
Combining \eqref{g3}, \eqref{g6}, and \eqref{g7}, we conclude that
$$\mathcal{C}_{(r-1)J+\ell}=M_{\ell,1,\infty}^{J}=N_{i,1,\infty}^{J}=\mathrm{HP}_i ^{2J+1},$$
where $\ell=r-i+1$.	This completes the proof.
\end{proof}

\section*{Acknowledgement}
We are grateful to the referee for carefully reading the manuscript and for many helpful comments. The third author is supported by the Natural Science Foundation of Hebei Province (A2023205045) and the 111 Center (D26018).

\end{document}